\documentclass[a4paper,11pt]{amsart}

\theoremstyle{plain}
\newtheorem{theorem}{Theorem}[section]
\newtheorem{lemma}[theorem]{Lemma}

\theoremstyle{remark}
\newtheorem{remark}[theorem]{Remark}

\theoremstyle{definition}
\newtheorem{definition}[theorem]{Definition}

\newcommand{\C}{\mathbf{C}}
\newcommand{\Q}{\mathbf{Q}}
\newcommand{\Z}{\mathbf{Z}}
\newcommand{\Ps}{\mathbf{P}}

\DeclareMathOperator{\sing}{sing}

\DeclareMathOperator{\van}{van}

\numberwithin{equation}{section}

\title[Noether--Lefschetz for singular threefolds]{Noether--Lefschetz theorem for hypersurface sections of singular threefolds}

\author[R.~Kloosterman]{Remke Kloosterman}
\address{Institut f\"ur Mathematik, Humboldt-Universit\"at zu Berlin,
Unter den Linden 6, D-10099 Berlin, Germany} 
\email{klooster@math.hu-berlin.de}

\thanks{The  author is partially supported by DFG-grant KL 2244/2-1.}
\begin{document}

\begin{abstract}
We prove a Noether--Lefschetz-type result for certain linear systems  on a projective threefold with isolated singularities.
\end{abstract}

\maketitle

\section{Introduction}
Let $X\subset \Ps^N$ be a smooth threefold and fix a degree $d\geq 1$. Let  $H$ be a very general hypersurface of  degree $d$ and let $X_H$ be the intersection $X\cap H$.
The Noether--Lefschetz theorem now states that either the Picard numbers or $X$ and $X_H$ coincide or the geometric genus of $X_H$ vanishes.

The aim of this paper is to extend this to the case where $X$ has isolated singularities. However, if $X$ is not $\Q$-factorial then the Picard numbers of $X$ and $X_H$ differ. To exclude such examples we will require  $h^4(X)=1$. We will make a further assumption in order to simplify our proof, namely we assume that $X$ admits a small resolution. However, we believe that with much more work one can avoid posing this condition. Our main result is

\begin{theorem}
 Let $X\subset \Ps^N$ be a threefold with isolated singularities, such that $h^4(X)=1$ holds. Suppose that $X$ admits a (non-projective) small resolution. Then for a very general hypersurface $H$ of degree $d\geq 1$ we have that  either $\rho(X_H)=1$ or $p_g(X_H)=0$ holds.
\end{theorem}

The strategy of the proof is similar to the classical proof of the Noether--Lefschetz theorem as one can find for example in \cite{Voi2}:
Let $Y=X_H$. Then we can consider $Y$ as a hypersurface in $X$ as well as a hypersurface in a small resolution $X'$ of $X$. On $X'$ we can
construct Lefschetz pencils  to produce vanishing cycles. We show that the monodromy acts transitively on the set of vanishing cycles. We also show that the subspace $H^2(Y)_{\van}$ generated by the vanishing cycles and $j^*H^2(X')$ generate $H^2(Y)$. From the properties of small resolution it follows easily that $h^2(X')=1$. Since the monodromy acts irreducibly on $H^2(Y)_{\van}$, it follows that the $H^2(Y)_{\van}$ cannot contain a non-trivial Hodge substructure. Therefore $H^2(Y)_{\van}$ is either of pure of type $(1,1)$ (and $p_g(Y)=0$) or does not contain a sub-Hodge structure of pure type $(1,1)$ and, in particular, the rank of $H^2(Y,\Z)\cap H^{1,1}(X,\C)$ is at most $h^2(X')$, which turns out to be one.

The main difficulty in extending the proof for Noether--Lefschetz theorem to our situation is to prove the existence of the decomposition $H^2(Y)=H^2(Y)_{\van}\oplus j^*H^2(X')$. The proof in \cite{Voi2}  uses the hard Lefschetz theorem to obtain this decomposition. The hard Lefschetz theorem requires $X'$ to be K\"ahler. However, if $X$ is singular and $X$ admits a small resolution, which is K\"ahler, then $h^4(X)>1$ holds.  Hence we cannot apply the hard Lefschetz theorem on $X'$. To avoid this problem we could  pass to the big resolution $\tilde{X}$ of $X$. However on $Y$ is not ample on $X$ and therefore we cannot apply the hard Lefschetz theorem on $\tilde{X}$.  Instead we give an ad hoc argument that the cup-product with the fundamental class of $Y$ defines an isomorphism $H^2(X')\to H^4(X')$ and that is where we use the assumption $h^4(X)=1$.

Our main motivation for this result lies in an application. In \cite{KloNod} we prove that a nodal compete intersection threefold with defect and without induced defect has at least $\sum_{i\leq j} (d_i-1)(d_j-1)$ nodes. In the proof we work with a nodal threefold satisfying $h^4(X)=1$ and we apply several times the main result of this paper.

\section{The proof}
We try to follow the proof from \cite{Voi2} as much as possible. We start by giving a preliminary result on small resolutions.
\begin{lemma} \label{lembe} Let $X$ be a projective threefold with isolated singularities admitting a small resolution $X'$. Then
 \[ h^2(X')=h^4(X')=h^4(X).\]
\end{lemma}
\begin{proof}
Consider the Mayer--Vietoris sequence of the square, i.e. the triangle $H^\bullet(X)\to H^{\bullet}(X')\oplus H^{\bullet}(\Delta) \to H^{\bullet}(E)$, where $E$ is the exceptional locus and $\Delta$ the singular locus of $X$ \cite[?]{PSbook}. Since $E$ is one-dimensional we obtain  $h^3(E)=h^4(E)=0$ and therefore that $H^4(X)\to H^4(X')$ is an isomorphism. 
Using Poincar\'e duality we get $h^2(X')=h^4(X')$.
\end{proof}
If $X\subset \Ps^N$ is smooth then the discriminant of $X$ (which is also the dual variety of $X$) is an irreducible variety. The corresponding result for the case of singular threefolds is as follows:

\begin{lemma}\label{lemDis} Let $X\subset \Ps^N$ be a projective threefold with isolated singularities. Then the discriminant $\Delta$ of $X$ in $(\Ps^N)^*$ is the union of $\#X_{\sing}$ hyperplanes together with one irreducible component $\Delta^0$.
\end{lemma}
\begin{proof}
 Let $p$ be a singular point of $X$. Then each hyperplane through $p$ is contained in the discriminant. Since these hyperplanes form a hyperplane in $(\Ps^N)^*$ they form an irreducible component $\Delta_p$ of $\Delta$.
 
 Consider next the set 
 \[ Z:=\{(x,H)\mid x\in X\setminus X_{\sing} \mbox{ and } x\in (X_H)_{\sing}\}\]
 Then the projection $Z\to X$  is a $\Ps^{N-4}$-bundle. In particular, $Z$ is irreducible. The projection of $Z$ to the second factor is $\Delta\setminus \cup_{p\in X_{\sing}} \Delta_p$. Hence there is precisely one irreducible component of $\Delta$ which is not contained in $\cup_{p\in X_{\sing}} \Delta_p$.
 \end{proof} 
We will use this result to construct  a Lefschetz pencil on $X'$:

Fix now a line $\ell \in (\Ps^N)^*$ such that $\ell$ intersects $\Delta$ transversally in its smooth locus. Moreover, if $\dim \Delta^0<N-1$ then $\ell$ does not intersect $\Delta^0$. In particular, $\ell$ avoids any intersection point of two irreducible components of $\Delta$. Now $\ell$ defines a one-parameter family of hyperplane sections $X_t$ of $X$. 
Consider now the pull-back $X'_t$ of $X_t$ to $X'$. 

\begin{lemma} We have that $X'_t$ is smooth if and only if $t\not \in \Delta^0\cap \ell$.
\end{lemma}
\begin{proof}
Suppose $t$ is such that  $X_t$ is contained in the smooth locus of $X$ then $X_t\cong X'_t$ and the statement is trivial

Suppose now that $t$ is such that  $X_t$ contains at least one of the singular points of $X$. Then  $t\in\Delta_p$ for some $p\in X_{\sing}$ and $p$ is a singular point of $X_t$. Since $\ell$ intersects $\Delta$ in its smooth locus we have  that $t\not \in \Delta_q$, for $q\neq p$ and that $t\not \in \Delta^0$. Hence $X_t$ is smooth away from $p$. The map $X'_t\to X_t$ is the blow-up of $p$ in $X_t$ and therefore $X'_t$ is smooth.
\end{proof}

\begin{lemma}
 The pencil $X'_t$ is a Lefschetz pencil.
\end{lemma}
\begin{proof}
 If $X'_t$ is singular then $t$ is smooth point of $\Delta^0$. This implies that $\ell$ intersects $\Delta^0$ and since $X_t$ is a Lefschetz pencil we obtain that $\dim \Delta^0=N-1$. Then the same reasoning as in \cite[Lemma 2.7, Corollary 2.8]{Voi2} yields that $X'_t$ has one singular point and that this point is a node. Therefore $X'_t$ is a Lefschetz pencil.
\end{proof}

 Let $\tilde{X}$ be a resolution of the map $X'\dashrightarrow \Ps^1$ induced by this pencil, i.e., the blow-up of the base locus of the pencil on $X'$. Let $U\subset \Ps^1$ be the locus of points with smooth fibers in $\tilde{X}$. Without loss of generality we may assume that $0,\infty\in U$ and that $X_0\cong X_0'$ and $X_{\infty}\cong X'_{\infty}$.
 
 \begin{lemma}\label{vangen}
Let $i_0:X_0\to \tilde{X}\setminus X_{\infty}$ be the inclusion. Then
\[ i_{0*}:H_k(X_0,\Z)\to H_k(\tilde{X}\setminus X_{\infty},\Z)\]
is an isomorphism for $k\leq 1$. The map $i_{0*}$ is surjective for $k=2$ with kernel  generated by the vanishing cycles. 
\end{lemma}
\begin{proof}
 If $X=X'$ (i.e., $X$ is smooth) then this follows from  \cite[Corollary 2.20]{Voi2}. 
  If $X$ is singular then $X'$ is not  K\"ahler, and therefore we cannot directly apply \cite[Corollary 2.20]{Voi2}. However, the proof can be extended to the  non-K\"ahler case.  Voisin shows first that $\tilde{X}\setminus X_{\infty}$ has the homotopy type of the union of $X_0$  with 3-dimensional balls glued along 2-dimensional balls. The proof of this uses several times Ehresmann's theorem and Morse theory and carries over to our case. Then the claim  follows from excision and the local version of this claim \cite[Corollary 2.17]{Voi2}.
 \end{proof}

\begin{lemma} Let $j:X_0\to X'$ be the inclusion. The kernel of $j_*:H^2(X_0,\Q)\to H^4(X',\Q)$ is generated by the vanishing cycles. 
\end{lemma}
\begin{proof}
Let $B$ be the base locus of the Lefschetz pencil $X'_t$. Then $H^4(\tilde{X})=H^4(X')\oplus H^2(B)$ and therefore $H_2(\tilde{X})=H_2(X')\oplus H_0(B)$. 

Note that the morphism  $j_*:H^2(X_0)\to H^4(X')$ is the Poincar\'e dual of $j_*:H_2(X_0)\to H_2(X')$. We can obtain this map by compsing $H_2(X_0)\to H_2(\tilde{X}\setminus X_{\infty})$ with the map $H_2(\tilde{X}\setminus X_{\infty}) \to H_2(\tilde{X})=H_2(X')\oplus H_0(B)$ and then projecting to the first factor.

The kernel of the first map is generated by the vanishing cycles by Lem\-ma~\ref{vangen}. The second map is injective by the same argument as in \cite[Corollary 2.23]{Voi2} and the map $H_2(X')\oplus H_0(B)\to H_2(X')$ is injective when restricted to the image of $H_2(X_0)$.
\end{proof}

\begin{lemma}All vanishing cycles are conjugated under the monodromy action.
 \end{lemma}
\begin{proof} 
The proof of \cite[Proposition 3.23]{Voi2} extends to our case: The only non-trivial thing to check is the irreducibility of the locus of hyperplanes $H$ such that the pull-back of $X\cap H$ to $X'$ is singular. But this locus is precisely the irreducible component $\Delta^0$ (Lemma~\ref{lemDis}).
\end{proof}

\begin{definition} Let $X\subset \Ps^n$ be threefold with isolated singularities and $H$ a hypersurface. Let $Y=X\cap Y$ and $i:Y\hookrightarrow X$ be the inclusion.
 Denote with $H^k(Y)_{\van}$ the kernel of $i_*:H^k(Y)\to H^{k+2}(X)$.
\end{definition}

\begin{lemma} If $h^4(X)=1$ holds then 
 \[ H^2(Y) = H^2(Y)_{\van}\oplus j^* H^2(X'). \] 
\end{lemma}
\begin{proof}
Note that we have $h^4(X')=h^2(X')=1$ by Lemma~\ref{lembe} and hence that $h^2(Y)_{\van}+h^2(X')=h^2(Y)$. Hence it suffices to prove  $H^2(Y)_{\van}\cap j^* H^2(X')=0$. So let $\beta\in H^2(X')$ be such that $j^*\beta\in H^2(Y)_{\van}$. In other words $j_*j^*\beta=0$. We claim now that $j_*j^*:H^2(X',\Q)\to H^4(X',\Q)$ is an isomorphism. Since both the domain and the target space are one-dimensional, it suffices to show that this map is non-zero. The map $j_*j^*$ is the cup product with $[Y]$. Since $[Y]^3$ is non-zero we have that $[Y]^2$ is also non-zero and therefore $j_*j^*([H])\neq 0$. Hence $j_*j^*$ is non-zero and therefore $\beta=0$.
\end{proof}
\begin{remark}
 In \cite{Voi2} the assumption $h^4(X)=1$ is not necessary. Voisin shows that $h^2(X')\to H^4(X')$ is an isomorphism by applying Hard Lefschetz. Since $X'$ is not K\"ahler we cannot apply this result directly.
\end{remark}

\begin{theorem}
 Let $X\subset \Ps^N$ be a threefold with isolated singularities, such that $h^4(X)=1$. Suppose that $X$ admits a (non-projective) small resolution. Then for a very general hypersurface $H$ of fixed degree $d\geq 1$ we have that either $\rho(X_H)=1$ holds or $X_H$ is a  surface with $p_g=0$. 
\end{theorem}
\begin{proof} 
 Since $H^2(X_H)_{\van}$ is generated by the vanishing cycles and the vanishing cycles are conjugated under the monodromy we have that the monodromy representation on $H^2(X_H)_{\van}$ is irreducible (cf. the proof of \cite[Theorem 3.27]{Voi2}). This implies that for a very general $X_H$ the Hodge structure of $H^2(X_H)_{\van}$ is irreducible (cf. \cite[Corollary 3.28]{Voi2}). Hence if for a very general $X_H$ we have $H^2(X_H)_{\van}\cap H^{1,1}\neq 0$ then $H^2(Y)$ is of pure $(1,1)$-type and therefore $X_H$ satisfies $p_g=0$.
\end{proof}
\bibliographystyle{plain}
\bibliography{remke2}

 \end{document}